\documentclass[12pt,reqno]{amsart}

\numberwithin{equation}{section}

\usepackage{txfonts}
\usepackage{amssymb}
\usepackage{amsmath}
\usepackage{amscd}
\usepackage{amsthm}
\usepackage{amsfonts}
\usepackage{ifpdf}
\ifpdf \usepackage[pdftex,pdfstartview=FitH,pdfpagemode=none,colorlinks,bookmarks,linkcolor=blue]{hyperref} \else  \usepackage[hypertex]{hyperref} \fi
\usepackage[nobysame,abbrev,alphabetic]{amsrefs}

\newtheorem{theorem}{Theorem}[section]
\newtheorem{proposition}[theorem]{Proposition}

\newtheorem{lemma}[theorem]{Lemma}

\theoremstyle{definition}
\newtheorem{definition}[theorem]{Definition}
\newtheorem{remark}[theorem]{Remark}

\def\mR {\mathbb {R}^n}

\newcommand{\dstyle}{\displaystyle}
\def\keywords#1{\par\medskip
\noindent\textbf{Keywords.} #1}

\begin{document}

\title[Isoperimetric inequality and Q-curvature]{Isoperimetric inequality, $Q$-curvature and $A_p$ weights}
\date{2013, April 23}
\author[Yi Wang]{Yi Wang}
\address{Yi Wang, Department of Mathematics, Stanford University, Stanford CA 94305, USA}
\address{ email: wangyi@math.stanford.edu}
\setcounter{page}{1}
\thanks{The research of the author is partially supported
 by NSF grant DMS-1205350}

\begin{abstract}
A well known question in differential geometry is to control the constant in isoperimetric inequality by intrinsic curvature conditions. In dimension 2, the constant can be controlled by the integral of the positive part of the Gaussian curvature. In this paper, we showed that on simply connected conformally flat manifolds of higher dimensions, the role of the Gaussian curvature can be replaced by the Branson's $Q$-curvature. We achieve this by exploring the relationship between $A_p$ weights and integrals of the $Q$-curvature.

\end{abstract}
\maketitle


\section{Introduction}\label{sect:intro}
The classical isoperimetric inequality on $\mathbb{R}^2$ states that for any bounded domain $\Omega\in \mathbb{R}^2$ with smooth boundary
$$ vol(\Omega)\leq \frac{1}{4\pi} Area(\partial \Omega)^2. $$
On a complete noncompact simply connected surface $M^2$, it is well-known that we have the Fiala-Huber \cite{Fiala}, \cite{Huber} isoperimetric inequality
\begin{equation}\label{FialaHuber}
vol(\Omega)\leq \frac{1}{2(2\pi-\int_{M^2}K_g^+ dv_g)} Area(\partial \Omega)^2,
\end{equation}
where $K_g^+$ is the positive part of the Gaussian curvature $K_g$. Also $\int_{M^2}K_g^+ dv_g< 2\pi$ is the sharp bound so that the isoperimetric inequality holds.

An important notion in conformal geometry is Branson's $Q$-curvature \cite{Branson} (called $Q$-curvature for short.)
In dimension 2, $Q_g=K_g/2$, and in dimension 4, $Q_g= \frac{1}{12}(-\Delta R_g+ \frac{1}{4}R_g^2- 3|E_g|^2)$, where $R_g$ denotes the scalar curvature and $E_g$ denotes the traceless part of the Ricci tensor.
However in general case $Q$-curvature remains a mysterious quantity that it is defined via analytic continuation in the dimension. (See the definition in section 2.) $Q$-curvature has conformal invariant properties analogous to the Gaussian curvature in dimension 2. There has been great progress in understanding it. For example see the work of \cite{FeffermanGraham}, \cite{GrahamZworski} on the study of the $Q$-curvature and ambient metrics; \cite{Spyros} on the structures of conformal invariants; and \cite{GurskyViaclovsky} on 4-manifolds that admit constant $Q$-curvature metrics.

Our study aims to understand another aspect of $Q$-curvature's geometric meaning--its relationship with isoperimetric inequality. In this paper, we prove the higher dimensional analogue of inequality (\ref{FialaHuber}) with $Q$-curvature bound.
\begin{theorem}\label{main}Let $(M^n,g)=(\mR, g= e^{2u}|dx|^2)$ be a complete noncompact even dimensional manifold.
Let $R_g$ denote the scalar curvature; $Q^+$ and $Q^-$ denote the
positive and negative part of $Q_g$ respectively; and $dv_g$ denote the volume form of $M$. Suppose $g= e^{2u}|dx|^2$ is a ``normal" metric, i.e.
\begin{equation}\label{normal}u(x)=
\displaystyle \frac{1}{c_n}\int_{\mR} \log \frac{|y|}{|x-y|} Q_{g}(y) dv_g(y) + C;
\end{equation}
for some constant $C$;
or instead, suppose $\liminf_{|x|\rightarrow \infty}R_g(x)\geq 0$.
If
\begin{equation}\label{assumption1}
\alpha:= \int_{M^n}Q^{+}dv_g < c_n
\end{equation}
where $c_n=2^{n-2}(\frac{n-2}{2})!\pi^{\frac{n}{2}}$,
and \begin{equation}\label{assumption2}
\beta:=\int_{M^n}Q^{-}dv_g < \infty,
\end{equation}
then $(M^n,g)$ satisfies the isoperimetric inequality with isoperimetric constant depending only on $n, \alpha$ and
$\beta$.
Namely, for any bounded domain $\Omega\subset M^n$ with smooth boundary,
\begin{equation}\label{1.89}
|\Omega|_g^{\frac{n-1}{n}}\leq C(n, \alpha,\beta) |\partial \Omega |_g.
\end{equation}
\end{theorem}
We remark that the constant $c_n$ in the assumption (\ref{assumption1}) is sharp. In fact, $c_n$ is equal to the integral of the $Q$-curvature
on a half cylinder (a cylinder with a round cap attached to one of its two ends); but obviously a half cylinder fails to satisfy the isoperimetric inequality. We also remark that being a normal metric is a natural and necessary assumption. We will give explanations of this remark in section 5 (Remark \ref{remarknormal}).

When one considers the isoperimetric inequality on higher dimensional manifolds, a natural question is to find suitable substitute of the Gaussian curvature. Many results \cite{Aubin}, \cite{Cantor}, \cite{Varopoulos}, \cite{Saloff} are obtained on proving the isoperimetric inequality with pointwise sectional curvature or Ricci curvature assumptions. Also a well-known conjecture asserts the validity of the Euclidean
isoperimetric inequality on complete simply connected manifolds with nonpositive sectional
curvature. This conjecture was proved in dimension 2 by Weil \cite{Weil}, in dimension 3
by Kleiner \cite{Kleiner}, and in dimension 4 by Croke \cite{Croke}; but it is still open for higher dimensions.

Instead of assuming the pointwise bound of some curvature, we ask a different and natural question that if there is a suitable curvature quantity, whose integral bound, controls the isoperimetric constant. In other words, we want to prove an inequality of type (\ref{FialaHuber}) in higher dimensions.
In the mean time, we know the conformal structure is always available on surfaces. Hence we ask the question that if the conformal structure is relevant in proving the isoperimetric inequality of type (\ref{FialaHuber}) on higher dimensions.
The main theorem of this paper answers both these questions.
We assert that the integral of the $Q$-curvature is the suitable curvature quantity to control the isoperimetric constant;
and the conformal structure is a key structure to look at in understanding the problem.

In the work of Chang, Qing and Yang \cite{CQY1} \cite{CQY2}, it was discussed that if the metric is ``normal" (as defined by (\ref{normal})) and $\int_{\mR} |Q_g|dv_g\leq \infty$, then the isoperimetric profile for very big balls is captured by the integral of the $Q$-curvature. This generalizes
the work of Cohn-Vossen \cite{Cohn-Vossen} and Huber \cite{Huber} for surfaces. More precisely, the relation is that
\begin{equation}\label{1.4}
\displaystyle \chi(\mathbb{R}^4)-\frac{1}{4\pi^2}\int_{\mathbb{R}^4} Q_g dv_g= \lim_{r\rightarrow
\infty }\frac{Area_g(\partial B(r))^{4/3}}{4(2\pi^2)^{1/3}Vol_g(B(r))},
\end{equation}
\noindent where $B(r)$ is the Euclidean ball with radius $r$.


Previous work was done by Bonk, Heinonen and Saksman \cite{BHS2}. They showed that if the metric is ``normal", and if in addition $\int_{\mR} |Q_g|dv_g\leq \epsilon_0$ for some small $\epsilon_0<<1$, the manifold is bi-Lipschitz to the Euclidean space, which in particular implies the isoperimetric inequality.
Also, in my previous work \cite{Wang}, we proved the isoperimetric inequality
with weaker assumptions than (\ref{assumption1}) and (\ref{assumption2}).
But the isoperimetric constant there does not only depend on $\alpha$ and $\beta$.

The main difficulty in the problem is to find analytical tools to characterize
different roles of the positive and negative part of the $Q$-curvature.
In this article, we adopt a very different method--the main proof relies on the theory of $A_p$ weights,
which is an important notion in harmonic analysis and potential theory.
Inspired by Peter Jones' result \cite{Jones} on the decomposition of $A_\infty$ weights; in particular, the idea of dyadic decomposition of BMO functions, we apply the theory of $A_p$ weights to handle the difficulty in this geometry problem.
Conceptually, the observation is that there is parallel structure between the geometric obstruction of having isoperimetric inequality with the analytic obstruction of being in suitable classes of $A_p$ weights.
We prove that the volume form $e^{nu}$ is a strong $A_\infty$ weight, and thus by a classical result of Guy David and Stephen Semmes \cite{DS1}, (see Theorem \ref{2.1} below), this implies the isoperimetric inequality is valid.


The paper will be organized as follows. In section 2, we present preliminaries on the $Q$-curvature and $A_p$ weights.
We then decompose the volume form $e^{nu}$ into two parts: $e^{nu_+}$ and $e^{nu_-}$ (see Definition \ref{defu+-}), and discuss them respectively in section 3 and 4.
In section 5 that we put these pieces together;
and show that $e^{nu}$ is a strong $A_\infty$ weight and finish the proof.


{\bf Acknowledgments:} The author would like to thank Alice Chang and Paul Yang for inspiring suggestions and discussions of this problem. She is very grateful to Matt Gursky for suggestions and comments. The author would also like to thank Mario Bonk for his continuous interest and encouragements on the project, and David Guy for valuable discussions.

\section{preliminaries}
\subsection{$Q$-curvature in conformal geometry}
In past decades, there are many works focusing on the study of the $Q$-curvature equation and the associated
conformal covariant operators, both from PDE point of view and from the geometry point of view.
We now discuss some background of it in conformal geometry.
Consider a 4-manifold $(M^4,g)$, the Branson's $Q$-curvature of $g$
is defined as
$$Q_g:=\frac{1}{12}\left\{-\Delta R_g +\frac{1}{4}R_g^2 -3|E|^2 ,\right\}
$$
where $R_g$ is the scalar curvature, $E_g$ is the traceless part of $Ric_g$, and $|\cdot|$ is taken with respect to the metric $g$.
It is well known that the $Q$-curvature is an integral conformal invariant associated to
the fourth order Paneitz operator $P_g$
$$P_g:=\Delta^2+\delta(\frac{2}{3}R_g g-2 Ric_g)d.$$
Under the conformal change $g_{u}=e^{2u}g_0$, $P_{g_u}=e^{-4u}P_{g_0}$,
$Q_{g_u}$ satisfies the fourth order differential equation,
\begin{equation}\label{1.88}P_{g_{0}}u+2Q_{g_0}=2Q_{g_{u}}e^{4u}.\end{equation}
This is analogous to the Gaussian curvature equation on surfaces
$$-\Delta_{g_0} u+ K_{g_0}  = K_{g_u} e^{2u}. $$
One particular situation is when the background metric $g_0=|dx|^2$. In this case, the equation (\ref{1.88}) reduces to
$$(-\Delta)^2 u= 2Q_{g_{u}}e^{4u}.$$
The invariance of the $Q$-curvature in dimension 4 is due to the Chern-Gauss-Bonnet formula for closed manifold $M^4$:
$$\chi(M^4)=\dstyle\frac{1}{4\pi^2} \int_{M^4}\left(\frac{|W_g|^2}{8}+Q_g\right) dv_{M},$$
where $W_g$ denotes the Weyl tensor.

For higher dimensions, the $Q$-curvature is defined via the analytic continuation in the dimension and the formula is not explicit in general.
However when the background metric is flat, it satisfies, under the conformal change of metric $g_{u}=e^{2u}|dx|^2$, the $n$-th order differential equation
$$(-\Delta)^{\frac{n}{2}} u= 2Q_{g_{u}}e^{nu}.$$
We will only use this property of $Q$-curvature in this paper.
\subsection{$A_p$ weights and Strong $A_\infty$ weights}
In this subsection, we are going to present the definitions and the properties of $A_p$ weights and strong $A_\infty$ weights.

In harmonic analysis, $A_p$ weights ($p\geq 1$) are introduced to characterize when a function $\omega$ could be a weight such that the associated measure $\omega(x)dx$
has the property that the maximal function $\textrm{M}$ of an $L^1$ function is weakly $L^1$, and that the maximal
function of an $L^p$ function is $L^p$ if $p>1$.

For a nonnegative locally integrable function $\omega$, we call it an $A_p$ weight $p>1$, if
\begin{equation}\label{Ap}
\frac{1}{|B|}\int_{B}\omega(x)dx\cdot \left(\frac{1}{|B|}\int_{B}\omega(x)^{-{p'}/{p}}dx\right)^{{p}/{p'}}\leq C<\infty,
\end{equation}
for all balls $B$ in $\mathbb{R}^n$. Here $p'$ is conjugate to $p$: $\frac{1}{p'}+\frac{1}{p}=1$. The constant $C$ is uniform for all $B$ and we call the smallest such constant $C$ the $A_p$ bound of $\omega$.
The definition of $A_1$ weight is given by taking limit of $p\rightarrow 1$ in (\ref{Ap}),
which gives
$$\dstyle \frac{1}{|B|} \int_{B}\omega dx \leq C \omega, $$
for almost all $x\in \mathbb{R}^n$.
Thus it is equivalent to say the maximal function of the weight is bounded by the weight itself:
$$\textsl{M}\omega(x)\leq C' \omega(x),$$
for a uniform constant $C'$.
Another extreme case is the $A_\infty$ weight. $\omega$ is called an $A_\infty$ weight if it is an $A_p$ weight for some $p>1$.
It is not difficult to see $A_1\subseteq A_p\subseteq A_p'\subseteq A_\infty$ when $1\leq p\leq p'\leq \infty$.

One of the most fundamental property of $A_p$ weight is the reverse H\"{o}lder inequality:
if $\omega$ is $A_p$ weight for some $p\geq 1$, then there exists an $r>1$ and a $C>0$, such that
\begin{equation}\label{reverse holder}
\dstyle \left( \dstyle \frac{1}{|B|} \int_{B}\omega^r dx \right)^{1/r}\leq \frac{C}{|B|} \dstyle \int_{B}\omega dx,
\end{equation}
for all balls $B$.
This would imply that any $A_p$ weight $\omega$ satisfies the doubling property:
there is a $C>0$ (it might be different from the $C$ in (\ref{reverse holder})), such that $$\dstyle \int_{B(x_0, 2r)} \omega(x) dx\leq C \int_{B(x_0, r)} \omega(x) dx$$
for all balls $B(x_0, r)\subset\mR$.

Suppose $\omega_1$ and $\omega_2$ are $A_1$ weights, and let $t$ be any positive real number. Then it is not hard to show $\omega_1\omega_2^{-t}$ is an $A_\infty$ weight.
Conversely, the factorization theorem of $A_\infty$ weight proved by Peter Jones \cite{Jones}: if $\omega$ is an $A_\infty$ weight, then there exist $\omega_1$ and $\omega_2$, both are $A_1$; and $t>1$ such that $\omega=\omega_1\omega_2^{-t}.$ Later, in the proof of the main theorem, we will decompose the volume form $e^{nu}$ into two pieces. The idea to decompose $e^{nu}$ is initially inspired by Peter Jones' factorization theorem.

Besides the standard $A_p$ weights, the notion of strong $A_\infty$ weight was first proposed by David and Semmes in \cite{DS1}.
Given a positive continuous weight $\omega$, we define $\delta_\omega(x,y)$ to be:
\begin{equation}\label{def1}
\delta_\omega(x,y):=\left(\int_{B_{x,y}}\omega(z)dz\right)^{1/n},
\end{equation}
where $B_{x,y}$ is the ball with diameter $|x-y|$ that contains $x$ and $y$.
One can prove that $\delta_\omega$ is only a quasi-distance in the sense that it satisfies the quasi-triangle inequality
$$\delta_\omega(x,y)\leq C(\delta_\omega(x,z)+ \delta_\omega(z,y)).$$
\noindent On the other hand, by taking infimum over all rectifiable arc $\gamma\subset B_{xy}$ connecting $x$ and $y$, one can define the $\omega$-distance to be
\begin{equation}\label{def2}d_\omega(x,y):=\inf_{\gamma}\int_\gamma \omega^{\frac{1}{n}}(s)|ds|.\end{equation}

If $\omega$ is an $A_\infty$ weight, then it is easy to prove (see for example Proposition 3.12 in \cite{S2})
\begin{equation}\label{A}
d_\omega(x,y)\leq C\delta_\omega(x,y)
\end{equation}
for all $x,y\in \mathbb{R}^n$.
If in addition to the above inequality, $\omega$ also satisfies the reverse inequality, i.e.
\begin{equation}\label{strong A}
\delta_\omega(x,y)\leq Cd_\omega(x,y),
\end{equation}
for all $x,y\in \mathbb{R}^n$, then we say $\omega$ is a strong $A_\infty$ weight.


Every $A_1$ weight is a strong $A_\infty$ weight, but for any $p>1$ there is an $A_p$ weight which is not strong $A_\infty$. Conversely, for
any $p>1$ there is a strong $A_\infty$ weight which is not $A_p$. It is easy to verify by definition the function $|x|^{\alpha}$ is $A_1$ thus strong $A_\infty$ if $-n<\alpha\leq 0$; it is not $A_1$ but still strong $A_\infty$ if $\alpha>0$.
And $|x_1|^\alpha$ is not strong $A_\infty$ for any $\alpha>0$ as one can choose a curve $\gamma$ contained in the
$x_2$-axis.

The notion of strong $A_\infty$ weight was initially introduced in order to study weights that are comparable to the Jacobian of quasi-conformal maps.
It was proved by Gehring that the Jacobian of a quasiconformal map on $\mR$ is always a strong $A_\infty$ weight, and it was conjectured that the converse was assertive: every strong $A_\infty$ weight is comparable to the Jacobian
of a quasi-conformal map. Later, however, counter-examples were found by Semmes \cite{S3} in dimension $n\geq 3$, and by Laakso \cite{La} in dimension 2.
Nevertheless, it was proved by David and Semme that a strong $A_\infty$ weight satisfies the Sobolev inequality:
\begin{theorem}\label{2.1}\cite{DS1}Let $\omega$ be a strong $A_\infty$ weight. Then for $f\in C^{\infty}_0(\mR)$,
\begin{equation}\dstyle \left( \int_{\mathbb{R}^n}|f(x)|^{p^*}\omega(x)dx\right)^{1/{p^*}}\leq C \dstyle \left(\int_{\mathbb{R}^n}
(\omega^{-\frac{1}{n}}(x)|\nabla f(x)|)^p\omega(x)dx \right)^{1/p},
\end{equation}
where $1\leq p<n$, $p^*=\frac{np}{n-p}$. Take $p=1$, it is the standard isoperimetric inequality.
\end{theorem}


By taking $f$ to be a smooth approximation of the indicator function of domain $\Omega$, this implies the validity of the isoperimetric inequality with respect to the weight $\omega$.
In this paper, we will take $\omega= e^{nu}$, the volume form of $(\mR,e^{2u}|dx|^2)$. We aim to show $e^{nu}$ is a strong $A_\infty$ weight. By Theorem \ref{2.1}, this implies the isoperimetric inequality on $(\mR,e^{2u}|dx|^2)$:
$$ \dstyle (\int_{\Omega} e^{nu(x)}dx)^{\frac{n-1}{n}}\leq C \int_{\partial \Omega} e^{(n-1)u(x)}d\sigma_x,$$
or equivalently, for $g=e^{2u}|dx|^2$,
$$|\Omega|_g^{\frac{n-1}{n}}\leq C |\partial \Omega|_g.$$

A good reference for $A_p$ weights is Chapter 5 in \cite{Stein}.
For more details on strong $A_\infty$ weight, we refer the readers to \cite{DS1}, where the concept was initially proposed.

\section{Analysis on the negative part of the $Q$-curvature}
We first remark that since $Q_g(y) e^{nu(y)}$ is integrable, $ \log \frac{|y|}{|x-y|} Q_g(y) e^{nu(y)}$ is also integrable in $y$ for each fixed $x\in \mR$. In fact, for a fixed $x$, the integral over the domain $|y|>>|x|$ is finite because $\log \frac{|y|}{|x-y|}$ is bounded and $Q_g(y) e^{nu(y)}$ is absolutely integrable by assumption (\ref{assumption1}) and (\ref{assumption2}); on the other hand, since the $Q$-curvature is smooth, and thus locally bounded, the integral over $B(x,1)$ is finite as well. Later in the paper, we will replace $Q_g(y)$ by either the positive or the negative part of it. The integral
$ \log \frac{|y|}{|x-y|} Q^{\pm}(y) e^{nu(y)}$ is still integrable for each fixed $x$. We will not repeat this point in the following sections.

To begin with, let us decompose $u=u_+ + u_-$, which are defined in the following.
\begin{definition}\label{defu+-}
 \begin{equation}
u_-(x):=\displaystyle \frac{-1}{c_n}\int_{\mR} \log \frac{|y|}{|x-y|} Q^-(y)e^{nu(y)} dy,
\end{equation}
and
 \begin{equation}
u_+(x):=\displaystyle \frac{1}{c_n}\int_{\mR} \log \frac{|y|}{|x-y|} Q^+(y)e^{nu(y)} dy.
\end{equation}
\end{definition}
In this section, we consider the negative part of the $Q$-curvature.
By (\ref{assumption2}), $ \beta:= \int_{\mR}Q^-(y)e^{nu(y)} dy<\infty$.
We recall the definitions (\ref{def1}), and (\ref{def2}), for a nonnegative function $\omega(x)$,
$$d_{\omega}(B_{xy}):= (\int_{B_{xy} } \omega(z)dz )^{\frac{1}{n}},
$$
$$\delta_{\omega}(x,y):=\inf_{\gamma} \displaystyle \int_{\gamma}\omega^{\frac{1}{n}}(\gamma(s))ds,
$$
where $B_{xy}$ is the ball with diameter $|x-y|$ that contains $x$ and $y$,
the infimum is taken over all curves $\gamma\subset B_{xy}$ connecting $x$ and $y$, and $ds$ is the arc length.

\begin{theorem}\label{Ainfty}$\omega^-(x):=e^{nu_-}$ is a strong $A_\infty$ weight,
i.e. there exists a constant $C=C(n,\beta)$ such that
\begin{equation}\label{Ainfty2}
\frac{1}{C(n,\beta)} d_{\omega^{-}}(B_{xy})
\leq \delta_{\omega^-}(x,y)\leq C(n,\beta) d_{\omega^{-}}(B_{xy}).
\end{equation}
\end{theorem}
We first observe that without generality we can assume $|x-y|=2$. This is because we can dilate $u$ by a factor $\lambda>0$,
\begin{equation}
\begin{split}
u^\lambda(x):= u(\lambda x)=&
\displaystyle \frac{-1}{c_n}\int_{\mR}\log \frac{|y|}{|\lambda x- y|} Q^-(y)e^{nu( y)} dy.\\
\end{split}
\end{equation}
By the change of variable, this is equal to
$$\displaystyle \frac{-1}{c_n}\int_{\mR}\log \frac{|y|}{|x-y|} Q^-(\lambda y)e^{nu(\lambda y)}\lambda^n dy.$$
Notice $Q^-(\lambda y)e^{nu(\lambda y)}\lambda^n$ is still an integrable function on $\mR$, with integral equal to $\beta$. Thus by choosing $\lambda =\frac{2}{|x-y|}$, the problem reduces to proving inequality (\ref{Ainfty2}) for $u^\lambda $ and $|x-y|=2$.


Let us denote the midpoint of $x$ and $y$ by $p_0$. And from now on, we adopt the notation
$\lambda B:=B (p_0,\lambda )$. Since $|x-y|=2$, we have $B_{xy}= B(p_0, 1)= B$.
We also define
\begin{equation}
u_1:=\displaystyle \frac{-1}{c_n}\int_{10B}\log \frac{|y|}{|x-y|} Q^-(y)e^{nu(y)} dy,
\end{equation}
and
\begin{equation}
u_2:=\displaystyle \frac{-1}{c_n}\int_{\mR\setminus 10 B} \log \frac{|y|}{|x-y|} Q^-(y)e^{nu(y)} dy.
\end{equation}

In the following lemma, we prove that when $z$ is close to $p_0$, the difference between $u_2(z)$ and $u_{2}(p_0)$ is controlled by $\beta$.
\begin{lemma}\label{claim1}
\begin{equation}\label{3.1}
|u_2(z)-u_{2}(p_0)|\leq \frac{ \beta}{4c_n}
\end{equation}
for $z\in 2 B $.
\end{lemma}

\begin{proof}
\begin{equation}
\begin{split}
&|u_2(z)-u_{2}(p_0)|\\
=& \frac{1}{c_n}\left|\int_{\mR\setminus 10 B}-\log \frac{|y|}{|z-y|} Q^-(y)e^{nu(y)}dy+  \int_{\mR\setminus 10 B}\log \frac{|y|}{|p_0-y|} Q^-(y)e^{nu(y)}dy\right |\\
=&\frac{1}{c_n}\left|\int_{\mR\setminus 10 B}\log \frac{|z-y|}{|p_0-y|} Q^-(y)e^{nu(y)}dy\right|\\
\leq &\frac{|z-p_0|}{c_n}\cdot \int_{\mR\setminus 10 B}\frac{1}{|(1-t^*)(p_0-y)+ t^*(z-y)|}Q^-(y)e^{nu(y)}dy, \\
\end{split}
\end{equation}
for some $t^*\in[0,1]$. Since $y\in \mR\setminus 10 B$ and $z, p_0\in 2B$,
$$\frac{1}{|(1-t^*)(p_0-y)+ t^*(z-y)|}\leq \frac{1}{8},$$
$|u_2(z)-u_{2}(p_0)|$ is bounded by
\begin{equation}
\begin{split}
\frac{|z-p_0|}{8 c_n }\cdot \int_{\mR\setminus 10 B}Q^-(y)e^{nu(y)}dy.
\end{split}
\end{equation}
Note that for $z\in 2B$, $|z-p_0|\leq 2 $. From this, (\ref{3.1}) follows.
\end{proof}

Now we adopt some techniques used in \cite{BHS}
for potentials to deal with the $\epsilon$-singular set $E_\epsilon$.
\begin{lemma}\label{Hauss1measure}(Cartan's lemma) For the Radon measure $Q^-(y)e^{nu(y)}dy$, given $\epsilon>0$, there exists a set $E_\epsilon\subseteq \mR$, such that
$$\mathcal{H}^1(E_\epsilon):= \dstyle \inf_{ E_\epsilon \subseteq \cup B_i }\{\dstyle \sum_{i}\mbox{diam } B_i \}< \epsilon $$
and for all $x\notin E_\epsilon$ and $r>0$,
$$\dstyle \int_{B(x,r)} Q^-(y)e^{nu(y)}dy\leq \frac{10 r \beta}{\epsilon }. $$
\end{lemma}
The proof of Lemma 1 follows from standard measure theory argument. Thus we omit it here.

\begin{proposition}\label{claim2} Given $\epsilon>0$,
$$  \mathcal{H}^1\left( \left\{x\in 10 B :\left |\frac{-1}{c_n} \int_{10 B}\log \frac{1}{|x-y|} Q^{-}(y)e^{nu}(y)dy  \right| >\frac{C_0\beta }{\epsilon}\right\}\right)< 10\epsilon,$$
for some $C_0$ depending only on $n$.
\end{proposition}
\begin{proof}
Fix $\epsilon>0$. By Lemma \ref{Hauss1measure}, there exists a set $E_\epsilon\subseteq \mR$, s.t.
$\mathcal{H}^1(E_\epsilon)< 10\epsilon $ and
for $x\notin E_\epsilon$ and $r>0$
\begin{equation}\label{3.3}\dstyle \int_{B(x,r)}Q^-(y) e^{nu(y)} dy \leq \frac{ r\beta}{\epsilon }.\end{equation}
If we can show for some $C_0= C_0(n)$
\begin{equation}\label{3.2}\dstyle 10 B \setminus E_\epsilon\subseteq \left\{ x\in 10 B :\left| \frac{-1}{c_n}\int_{10 B}\log \frac{1}{|x-y|}Q^-(y)e^{nu(y)}dy\right| \leq \frac{C_0}{\epsilon} \beta\right\}, \end{equation}
then
$$\mathcal{H}^1\left( \left\{x\in 10 B :\left |\frac{-1}{c_n} \int_{10 B}\log \frac{1}{|x-y|} Q^{-}(y)e^{nu}(y)dy  \right| >\frac{C_0\beta }{\epsilon}\right\}\right)\leq  \mathcal{H}^1(E_\epsilon)<
10\epsilon.$$
To prove (\ref{3.2}), we notice for $x\in 10 B\setminus E_\epsilon$, $r= 2^{-j}\cdot 10$, (\ref{3.3}) implies
\begin{equation}
\begin{split}
&\left|\dstyle \frac{-1}{c_n}\int_{10 B} \log \frac{1}{|x-y|} Q^-(y)e^{nu(y)} dy\right|\\
\leq &\frac{1}{c_n} \dstyle \sum_{j=-1}^\infty   \left|\int_{B(x, 2^{-j}\cdot 10)\setminus B(x, 2^{-(j+1)}\cdot 10)} \log \frac{1}{|x-y|} Q^-(y)e^{nu(y)} dy\right|\\
\leq&\dstyle \frac{1}{c_n}  \sum_{j=-1}^\infty \left(\max\{|\log 2^{-j}|,|\log 2^{-(j+1)}|\} + \log 10\right) \cdot \int_{B(x, 2^{-j}\cdot 10)\setminus B(x, 2^{-(j+1)}\cdot 10)} Q^-(y)e^{nu(y)}dy\\
\leq &\dstyle \frac{1}{c_n}  \sum_{j=-1}^\infty \left(\max\{|\log 2^{-j}|,|\log 2^{-(j+1)}|\} + \log 10\right) \cdot  \frac{2^{-j}\cdot 10 \beta}{\epsilon }\\
\leq & \frac{C_0\beta}{\epsilon },\\
\end{split}
\end{equation}
where $$C_0=\frac{ 10 \sum_{j=-1}^\infty \left(\max\{|\log 2^{-j}|,|\log 2^{-(j+1)}|\}+ \log 10\right) \cdot 2^{-j}}{c_n}<\infty, $$
depending only on $n$.
This completes the proof of the proposition.
\end{proof}
We next estimate the integral of $e^{nu_-(z)} $ over $2B$.
\begin{proposition}\label{3.8}Let $\bar{c}:= \frac{-1}{c_n}\int_{10B} \log |y|Q^-(y)e^{nu(y)}dy$. $\bar{c}<\infty$, since $Q^-(y)e^{nu(y)}$ is continuous thus bounded near the origin. Then
\begin{equation}\label{3.4}
\int_{2B} e^{nu_-(z)}dz\leq C_1(n,\beta) e^{nu_2(p_0)} e^{n\bar{c}},
\end{equation}
for $C_1(n,\beta)= e^{\frac{n\beta}{4c_n}}  12^{\frac{n\beta_{10}}{c_n}} \frac{\omega_{n-1}2^n}{n}$, where $\omega_{n-1}$ denotes the area of the (n-1)-dimensional unit sphere in $\mR$
and $\beta_{10}:=\int_{10B} Q^{-}(y)e^{nu(y)} dy\leq \beta<\infty $.
\end{proposition}

\begin{proof}Recall
\begin{equation}
u_1:=\displaystyle \frac{-1}{c_n}\int_{10B}\log \frac{|y|}{|x-y|} Q^-(y)e^{nu(y)} dy,
\end{equation}
and
\begin{equation}
u_2:=\displaystyle \frac{-1}{c_n}\int_{\mR\setminus 10 B} \log \frac{|y|}{|x-y|} Q^-(y)e^{nu(y)} dy.
\end{equation}
By Lemma \ref{claim1},
\begin{equation}\label{3.12}
\begin{split}
\int_{2B} e^{nu_-(z)}dz=&\int_{2B} e^{nu_1(z)}e^{nu_2(z)}dz\\
\leq&  e^{\frac{n\beta}{4c_n}}e^{nu_2(p_0)}\int_{2B} e^{nu_1(z)}dz.\\
\end{split}
\end{equation}

To estimate $u_1$, by definition
$\beta_{10}:=\int_{10B} Q^{-}(y)e^{nu(y)} dy\leq \beta<\infty $.
If $\beta_{10}= 0$, then $u_1(z)=0$ and $\bar{c}:= \frac{-1}{c_n}\int_{10B} \log |y|Q^-(y)e^{nu(y)}dy=0$. So (\ref{3.4}) follows immediately. If $\beta_{10}\neq 0$,
$\frac{Q^{-}(y)e^{nu(y)}}{\beta_{10}} dy$ is a nonnegative probability measure on $10 B$. Hence by Jensen's inequality
\begin{equation}
\begin{split}
\int_{2B} e^{nu_1(z)}dz= &e^{n\bar{c}}\cdot \int_{2B} e^{\frac{n}{c_n} \int_{10B}\log |z-y|Q^{-}(y)e^{nu(y) } dy } dz\\
\leq &e^{n\bar{c}}\cdot \int_{2B}  \int_{10B}|z-y|^{\frac{n\beta_{10}}{c_n}} \frac{Q^{-}(y)e^{nu(y)} }{\beta_{10}}dy  dz.\\
\end{split}
\end{equation}
Since $z\in 2B$ and $y\in 10B$,
\begin{equation}
 \int_{2B}|z-y|^{\frac{n\beta_{10}}{c_n}}dz\leq 12^{\frac{n\beta_{10}}{c_n}} \frac{\omega_{n-1}2^n}{n}.
\end{equation}
From this, we get
\begin{equation}
\begin{split}
\int_{2B} e^{nu_1(z)}dz\leq &  e^{n\bar{c}} 12^{\frac{n\beta_{10}}{c_n}}  \frac{\omega_{n-1}2^n}{n}  \int_{10B} \frac{Q^{-}(y)e^{nu(y)}}{\beta_{10}}dy\\
= &  e^{n\bar{c}} 12^{\frac{n\beta_{10}}{c_n}}   \frac{\omega_{n-1}2^n}{n}.\\
\end{split}
\end{equation}
Plugging it to (\ref{3.12}), we finish the proof of the proposition.
\end{proof}
Now we are ready to prove Theorem \ref{Ainfty}.\\
\noindent{\it{Proof of Theorem \ref{Ainfty}.}}
Let us assume $e^{nu_-}$ is an $A_p$ weight for some large $p$, with bounds depending only on $n$ and
$\beta$. We will prove this statement, in fact for a more general setting, in Proposition \ref{5.2}.
So by the reverse H\"{o}lder's inequality for $A_p$ weights, it is easy to prove (see for example Proposition 3.12 in \cite{S2}),
$$\delta_{\omega^-}(x,y)\leq C_2(n,\beta) d_{\omega^-}(x,y).$$
Hence we only need to prove the other side of the inequality:
\begin{equation}\label{3.9}\delta_{\omega^-}(x,y)\geq C_3(n,\beta) d_{\omega^-}(x,y),
\end{equation}
for some constant $C_3(n,\beta)$.
By Proposition \ref{claim2}, for a given $\epsilon>0$, there exists a Borel set $E_\epsilon\subseteq
\mR$, such that
\begin{equation}\label{3.13}\mathcal{H}^{1}(E_\epsilon)\leq 10 \epsilon,\end{equation} and for
$z\in 10B \setminus E_\epsilon$, according to (\ref{3.2})
\begin{equation}
\label{3.7}
|\hat{u}_1(z)|\leq \frac{C_0}{\epsilon}\beta. \end{equation}
Here $$\hat{u}_1(z):=\dstyle \frac{-1}{c_n}\int_{10 B}\log \frac{1}{|x-y|} Q^{-}(y)e^{nu(y)}dy.  $$
With this, we claim the following estimate.\\
{{\bf Claim:} Suppose $\mathcal{H}^{1}(E_\epsilon)< 10\epsilon$ with $\epsilon\leq \frac{1}{20}$. Then
\begin{equation}\label{arclength-est}
\dstyle \mbox{length } (\gamma\setminus E_\epsilon)> \frac{3}{2},
\end{equation}
where $\gamma\subset B_{xy}$ is a curve connecting $x$ and $y$.
}

{\it {Proof of Claim.}}
Let $P$ be the projection map from points in $B_{xy}$ to the line segment $I_{xy}$ between $x$ and $y$. Since the Jacobian of projection map is less or equal to 1,
\begin{equation}\label{3.5}
\mbox{length } (\gamma\setminus E_\epsilon)\geq
\mbox{length } (P(\gamma\setminus E_\epsilon))= m(P(\gamma\setminus E_\epsilon)),
\end{equation}
where $m$ is the Lebesgue measure on the line segment $I_{xy}$.
Notice $P(\gamma)=I_{xy}$, and
$P(\gamma)\setminus P(E_\epsilon)$ is a subset of $P(\gamma\setminus E_\epsilon)$. Therefore
\begin{equation}\label{3.6}
m(P(\gamma\setminus E_\epsilon))\geq m(P(\gamma))- m(P(E_\epsilon))=2-m(P(E_\epsilon)).\end{equation}
Now by assumption, $\mathcal{H}^{1}(E_\epsilon)< 10\epsilon$,
 so $\mathcal{H}^{1}(\gamma\cap E_\epsilon)< 10\epsilon$. Hence
there is a covering $\cup_i B_i$ of $\gamma \cap E_\epsilon$, so that
$$ \dstyle \sum_{i}\mbox{diam } B_i< 10\epsilon. $$
This implies that $\cup_{i} P(B_{i})$ is a covering of the set $P(\gamma \cap E_\epsilon)$
and
$$\dstyle \sum_{i}\mbox{diam } P(B_i)= \dstyle \sum_{i}\mbox{diam } B_i\leq 10\epsilon. $$
Thus $m(P(E_\epsilon))= \mathcal{H}^1(P(E_\epsilon))< 10\epsilon< \frac{1}{2}$, by choosing $\epsilon \leq \frac{1}{20}$. Plug it to (\ref{3.6}), and then to (\ref{3.5}). This completes the proof of the claim.

We now continue the proof of Theorem \ref{Ainfty}.
Since $\gamma \subset B$, then by Lemma \ref{claim1},
\begin{equation}
\begin{split}
\dstyle \int_{\gamma} e^{u_-(\gamma(s))} ds=\int_{\gamma} e^{(u_1+u_2)(\gamma(s))} ds \geq &\dstyle e^{\frac{-\beta}{4c_n}}e^{u_{2}(p_0)}  e^{\bar{c}}  \int_{\gamma}  e^{\hat{u}_1(\gamma(s))}ds.\\
\end{split}
\end{equation}
Here $\bar{c}$ is the constant defined in Proposition \ref{3.8}.
Let $\epsilon =\frac{1}{20}$. By (\ref{3.7}),
$$|\hat{u}_1(z)|\leq 20C_0 \beta $$
for $z\in 10B\setminus E_\epsilon$.
Thus
\begin{equation}
\dstyle \int_{\gamma}  e^{\hat{u}_1(\gamma(s))}ds\geq e^{- 20 C_0 \beta} \mbox{length } (\gamma\setminus E_\epsilon) .\\
\end{equation}
By (\ref{arclength-est}), it is bigger than
$$\frac{3}{2}e^{- 20 C_0 \beta}. $$
Therefore
\begin{equation}\label{3.10}
\begin{split}
\dstyle \int_{\gamma} e^{u_-(\gamma(s))} ds\geq
\frac{3}{2} e^{\frac{-\beta}{4c_n}} e^{- 20  C_0 \beta}e^{u_{2}(p_0)}  e^{\bar{c}}=C_4(n,\beta)e^{u_{2}(p_0)}  e^{\bar{c}}
\end{split}
\end{equation}
for $C_4(n,\beta)=\frac{3}{2} e^{\frac{-\beta}{4c_n}} e^{- 20  C_0 \beta}$.
By inequality (\ref{3.10}) and Proposition \ref{3.8}, we conclude for any curve $\gamma\subset B$ connecting $x$ and $y$, there is a $C_3=C_3(n,\beta)$ such that
\begin{equation}
\begin{split}
\dstyle \int_{\gamma} e^{u_-(\gamma(s))} ds\geq
C_{3}(n,\beta) (\int_{B_{xy}} e^{nu_-(z)}dz)^{\frac{1}{n}}.
\end{split}
\end{equation}
This implies inequality (\ref{3.9}) and thus completes the proof of Theorem \ref{Ainfty}.

\section{On the positive part of the $Q$-curvature}
In this section, we consider the positive measure $\frac{1}{c_n}Q^+(x) e^{nu(x)} dx$. We recall the assumption (\ref{assumption1}), $\alpha:=\int_{\mR} Q^+(x) e^{nu(x)} dx< c_n $. Recall Definition \ref{defu+-},
\begin{equation}
u_+(x):=\displaystyle \frac{1}{c_n}\int_{\mR} \log \frac{|y|}{|x-y|} Q^+(y)e^{nu(y)} dy.
\end{equation}
\begin{theorem}\label{A1}Suppose
$\alpha:=\int_{\mR} Q^+(x) e^{nu(x)} dx< c_n $. Then
$e^{nu_+}$ is an $A_1$ weight, i.e.
\begin{equation}\label{eqnA1}
M(e^{nu_+})(x) \leq  C(n,\alpha) e^{nu_+(x)} \quad \mbox{a.e.} \quad x\in \mR,
\end{equation}
where $M(\cdot)$ denotes the maximal function $$M(f)(x):= \sup_{r>0}\fint_{B(x, r)} |f(y)| dy.$$
\end{theorem}
\noindent{\it{Proof of Proposition \ref{A1}:}} Note that
\begin{equation}
\begin{array}{lcl}\label{73}
\dstyle \frac{\textsl{M}(e^{nu_+})(x)}{e^{nu_+(x)} }&=& \dstyle \sup_{r>0}\frac{\frac{1}{|B(x,r)|}
\dstyle \int_{B(x,r)} \exp\left( \dstyle \frac{n}{c_n}\int_{\mR}\log\frac{|z|}{|y-z|} Q^+(z)e^{nu(z)}dz \right)dy }{ \exp\left(\dstyle \frac{n}{c_n}\int_{\mR}\log\frac{|z|}{|x-z|} Q^+(z)e^{nu(z)}dz\right)}\\
&=&\dstyle \sup_{r>0}\frac{1}{|B(x,r)|}
\dstyle\int_{B(x,r)}\exp\left(\dstyle \frac{n}{c_n}\int_{\mR}\log\frac{|x-z|}{|y-z|} Q^+(z)e^{nu(z)}dz \right)dy\\
\end{array}
\end{equation}
If $\alpha=0$, then (\ref{eqnA1}) is obviously true. So let us assume $\alpha \neq 0$ and define the nonnegative probability measure $\nu_+(z):=\frac{Q^+(z)e^{nu(z)}dz}{\alpha}$. By
Jensen's inequality,
we get for any $r>0$,
\begin{equation}\label{74}
\begin{array}{lcl}
&&\dstyle\frac{1}{|B(x,r)|}
\int_{B(x,r)}\exp\left(\dstyle \frac{n}{c_n}\int_{\mR}\log\frac{|x-z|}{|y-z|} Q^+(z)e^{nu(z)}dz \right)dy\\
&\leq & \dstyle\frac{1}{|B(x,r)|}
\int_{B(x,r)}\dstyle \int_{\mR}\left(\frac{|x-z|}{|y-z|}\right)^{\frac{n\alpha}{c_n}} d\nu_+(z)\ dy\\
&=&\dstyle \int_{\mR}\frac{1}{|B(x,r)|}
\int_{B(x,r)} \left(\frac{|x-z|}{|y-z|}\right)^{\frac{n\alpha}{c_n}}  dy \ d\nu_+(z).\\
\end{array}
\end{equation}
As discussed in section 2, $\frac{1}{|x|^{\frac{n\alpha}{c_n}}}$ is an $A_1$ weight on $\mR$ with $A_1$ bound depending on $n$ and $\alpha$ when $\alpha<c_n$.
Hence for any $x\in\mR$, $r>0$,
\begin{equation}\dstyle \frac{\frac{1}{|B(x,r)|}
\dstyle \int_{B(x,r)} \left(\frac{1}{|y|^{\frac{n\alpha}{c_n}}}\right)  dy}{\dstyle \frac{1}{|x|^{\frac{n\alpha}{c_n}}}}\leq C(n,\alpha),
\end{equation}
Obviously if we shift the function $\frac{1}{|x|^{\frac{n\alpha}{c_n}}}$ by any point $z\in \mR$, the inequality is still valid with the same constant $C(n,\alpha)$, i.e.
\begin{equation} \dstyle \frac{\frac{1}{|B(x,r)|}
\dstyle \int_{B(x,r)} \left(\frac{1}{|y-z|^{\frac{n\alpha}{c_n}}}\right)  dy}{\dstyle \frac{1}{|x-z|^{\frac{n\alpha}{c_n}}}}\leq C(n,\alpha).
\end{equation}
Applying it to (\ref{74}), we obtain
\begin{equation}
\begin{array}{lcl}\label{77}
&&\dstyle \int_{\mR}\frac{1}{|B(x,r)|}
\int_{B(x,r)} \left(\frac{|x-z|}{|y-z|}\right)^{\frac{n\alpha}{c_n}}  dy \ d\nu_+(z)\\
&\leq & \dstyle \int_{\mR}  C(n,\alpha)\ d\nu_+(z) =  C(n,\alpha),\\
\end{array}
\end{equation}
for any $r>0$ and $x\in \mR$.
Thus (\ref{eqnA1}) follows.
This finishes the proof of the theorem.

\section{Proof of Theorem \ref{main}}
We begin this section by showing that $e^{nu}$ is an $A_p$ weight for large $p$.
\begin{proposition}\label{5.2}For
\begin{equation}u(x)=
\displaystyle \frac{1}{c_n}\int_{\mR} \log \frac{|y|}{|x-y|} Q(y)e^{nu(y)} dy
\end{equation}
with assumptions (\ref{assumption1}) and (\ref{assumption2}),
$e^{nu(x)}$ is an $A_p$ weight for some large $p$. Its $A_p$ bound depends only on $n$, $\alpha$ and $\beta$.
\end{proposition}
\begin{proof}
By Theorem \ref{A1}, $e^{nu_+}$ is an $A_1$ weight, so there is a uniform constant $C=C(n, \alpha)$, so that for all $x_0\in \mR$ and $r>0$
\begin{equation}
\dstyle \frac{1}{|B(x_0,r)|} \int_{B(x_0,r)} e^{nu_+(y)} dy \leq C(n,\alpha) e^{nu_+(x_0)}.
\end{equation}
So for all $x\in B(x_0,r)$
\begin{equation}
\begin{split}
 \dstyle \frac{1}{|B(x_0,r)|} \int_{B(x_0,r)} e^{nu_+(y)} dy \leq& \dstyle \frac{1}{|B(x_0,r)|} \int_{B(x,2r)} e^{nu_+(y)} dy\\
&=  \frac{2^n}{|B(x,2r)|} \int_{B(x,2r)} e^{nu_+(y)} dy\\
&\leq 2^n C(n,\alpha) e^{nu_+(x)}.\\
\end{split}
\end{equation}
Namely, for all ball $B$ in $\mR$ and $x\in B$,
\begin{equation}\label{5.3}
 \dstyle \frac{1}{|B|} \int_{B} e^{nu_+(y)} dy \leq 2^n C(n,\alpha) e^{nu_+(x)}.
\end{equation}
We observe that $e^{ -\epsilon n u_-(x)}  $ is also an $A_1$ weight for $\epsilon= \epsilon (n, \beta)<<1$.
In fact,
\begin{equation}e^{ -\epsilon n u_-(x)}   = e^{ \frac{n}{c_n}\int_{\mR} \log\frac{|y|}{|x-y|} \epsilon Q^{ -}(y) e^{nu(y)} dy   }. \end{equation}
$ Q^{ -}(y) e^{nu(y)}\geq 0 $ and $\int_{\mR} \epsilon Q^{ -}(y) e^{nu(y)} dy< c_n $ if $\epsilon$ is small enough.
Thus by Theorem \ref{A1}, $e^{ -\epsilon n u_-(x)}  $ is an $A_1$ weight,
As (\ref{5.3}), we have
\begin{equation}\label{5.4}
\dstyle \frac{1}{|B|} \int_{B} e^{ -\epsilon n u_-(y)}     dy \leq C(n,\beta) e^{ -\epsilon n u_-(x)}
\end{equation}
for all ball $B$ in $\mR$ and all $x\in B$.
Choose $1<p<\infty$ such that $\epsilon= p'/p$ with $\frac{1}{p}+ \frac{1}{p'}= 1$.
Using $ e^{nu}= e^{nu_+}\cdot e^{nu_-} $, we get
\begin{equation}\label{5.16}
\begin{split}
&\left(\int_{B}  e^{nu(x)} dx\right)^{\frac{1}{p}}\left(\int_{B} (e^{nu(x)})^{-\frac{p'}{p}} dx\right)^{\frac{1}{p'}}\\
= &\left(\int_{B}  e^{nu_+}\cdot (e^{-\epsilon n u_-} )^{-\frac{1}{\epsilon}}   dx\right)^{\frac{1}{p}} \left(\int_{B} (e^{ nu_+})^{-\frac{p'}{p}} \cdot e^{-\epsilon n u_-}   dx\right)^{\frac{1}{p'}}.\\
\end{split}
\end{equation}
By (\ref{5.4}), if $p$ is large enough and thus $\epsilon$ is small enough, then
$$(e^{-\epsilon n u_-} )^{-\frac{1}{\epsilon}} \leq  \left(\frac{1}{C(n,\beta) |B|}\int_{B}e^{-\epsilon n u_-}dx\right)^{-\frac{1}{\epsilon }}.$$
So
\begin{equation}\label{5.14}
\begin{split} \left(\int_{B}  e^{nu_+}\cdot (e^{-\epsilon n u_-} )^{-\frac{1}{\epsilon}}   dx\right)^{\frac{1}{p}}
\leq& \left(\int_{B}  e^{nu_+}    dx\right)^{\frac{1}{p}}   \left(\frac{1}{C(n,\beta) |B|}\int_{B}e^{-\epsilon n u_-}dx\right)^{-\frac{1}{\epsilon p}} \\
 =& \left(\int_{B}  e^{nu_+}    dx\right)^{\frac{1}{p}}   \left(\frac{1}{C(n,\beta)  |B|}\int_{B}e^{-\epsilon n u_-}dx\right)^{-\frac{1}{p'}}.\\
\end{split}
\end{equation}
Similarly, by (\ref{5.3})
 $$(e^{ nu_+})^{-\frac{p'}{p}}\leq \left(\frac{1}{2^nC(n,\alpha)  |B|}\int_{B} e^{ nu_+} dx \right)^{-\frac{p'}{p}}.$$
So \begin{equation}\label{5.15}
\left(\int_{B} (e^{ nu_+})^{-\frac{p'}{p}} \cdot e^{-\epsilon n u_-}  dx\right)^{\frac{1}{p'}}
\leq
\left(\frac{1}{2^nC(n,\alpha)   |B|}\int_{B} e^{ nu_+} dx \right)^{-\frac{1}{p}}
\left(\int_{B}  e^{-\epsilon n u_-} dx\right)^{\frac{1}{p'}}.
 \end{equation}
Applying (\ref{5.14}) to (\ref{5.15}) in (\ref{5.16}), we have
\begin{equation}
\left(\int_{B}  e^{nu(x)} dx\right)^{\frac{1}{p}}\left(\int_{B} (e^{nu(x)})^{-\frac{p'}{p}} dx\right)^{\frac{1}{p'}}\\
\leq (\frac{1}{C |B|})^{-\frac{1}{p}-\frac{1}{p'}}=C|B|
\end{equation}
for $p>>1$.
This shows that $ e^{nu(x)}$ is an $A_p$ weight for $p>>1$. The bound $C$ depends only on $n$, $\alpha$ and $\beta$.
\end{proof}

Now we recall a lemma \cite[Lemma 3.17]{S2}. Though the set-up in \cite{S2} is slightly different (but equivalent) from ours;
the proof of the lemma is straightforward following the definitions of strong $A_\infty$ weight and $A_1$ weight. Thus we omit it here.
\begin{lemma}\label{5.1}\cite[Lemma 3.17]{S2}
Assume that $\omega_1$ is an $A_1$ weight, $\omega_2$ is a strong $A_\infty$ weight, and that $r$ is a positive real number. If $\omega_1^r \omega_2$ is $A_\infty$, then $\omega_1^r\omega_2$ is strong $A_\infty$.
\end{lemma}

\noindent {\it{Proof of Theorem \ref{main}:}}
In Theorem \ref{A1} and Theorem {\ref{Ainfty}} we have proved that $e^{nu_+(x)}$ is an $A_1$ weight and $e^{nu_-(x)}$ is a strong $A_\infty$ weight.
Also by Proposition \ref{5.2}, $e^{nu}= e^{nu_+}\cdot e^{nu_-}$ is an $A_p$ weight for $p>>1$. Therefore $ e^{nu(x)}$ is an $A_\infty$ weight. Applying Lemma \ref{5.1} (with $r=1$), we obtain $e^{nu}$ is a strong $A_\infty$ weight with bound depending only on $n$, $\alpha$ and $\beta$. Therefore according to Theorem \ref{2.1}, the isoperimetric inequality is valid with constant depending only on the bound of strong $A_\infty$ weight of $e^{nu}$, and thus only on $n$, $\alpha$ and $\beta$. This completes the proof of Theorem \ref{main}.

\begin{remark}\label{remarkc_n}As we pointed out in the introduction, the assumption (\ref{assumption1}) is sharp. In fact, $c_n$ is equal to the integral of the $Q$-curvature of the standard sphere metric on a unit hemisphere, and $Q$-curvature is equal to 0 on a flat cylinder. Thus a cylinder with a hemisphere attached to one of its end (one can slightly perturb the metric in order to glue smoothly) has $\alpha=c_n$ and $\beta=0$; and it is conformal equivalent to $(\mR, |dx|^2)$. But such a manifold certainly fails to satisfy the isoperimetric inequality.
\end{remark}
\begin{remark}\label{remarknormal}The assumption on ``normal metric" is necessary when dimension is higher than 2, due to the nature of the problem.
On one hand, if this assumption is removed, there are examples of manifolds with non-uniform isoperimetric constant;
on the other hand, no assumption on ``normal metric" is needed when $n=2$. Because by \cite{Huber}'s result, every complete noncompact metric with integrable Gaussian curvature is ``normal". So the assumption is implicit when $n=2$.
\end{remark}

\begin{remark}In the statement of Theorem \ref{main}, we also mentioned the assumption $\liminf_{|x|\rightarrow \infty}R_g(x)\geq 0$ can replace
the assumption on ``normal metric". This is because by a maximum principal argument, $\liminf_{|x|\rightarrow \infty}R_g(x)\geq 0$ implies the metric is normal. See for example \cite[Theorem 4.1]{CQY1} for the proof.
\end{remark}
\begin{remark}In fact, by a similar argument, one can even show $e^{nu}$ is a stronger $A_\infty$ weight (see \cite[Definition 5.1]{S2} for the definition), which is a stronger conclusion than being a strong $A_\infty$ weight.
 However, for the purpose of the present paper, there is no need to get into the details of this point.
\end{remark}

\end{document}